\documentclass[english,11pt,reqno,twoside]{amsart}

\usepackage{amssymb,amsmath,amsthm,soul,color}
\usepackage{graphicx}
\usepackage[noadjust]{cite}
\usepackage{t1enc}
\usepackage[utf8]{inputenc}
\usepackage{braket}
\usepackage[T1]{fontenc}
\usepackage{amsfonts}
\usepackage{a4,indentfirst,latexsym,graphics}
\usepackage[english]{babel}
\usepackage{mathtools}
\usepackage{fixmath}

\usepackage[left=2.1cm,right=2.1cm,top=2cm,bottom=2cm]{geometry}

\usepackage{url}

\usepackage{mathrsfs}

\usepackage{enumitem}
\usepackage[colorlinks=true,urlcolor=blue,
citecolor=red,linkcolor=blue,linktocpage,pdfpagelabels,
bookmarksnumbered,bookmarksopen]{hyperref}

\usepackage[colorinlistoftodos]{todonotes}

\usepackage[normalem]{ulem}

\usepackage{lmodern}

\newtheorem{thm}{Theorem}[section]
\newtheorem{cor}[thm]{Corollary}
\newtheorem{lem}[thm]{Lemma}
\newtheorem{prop}[thm]{Proposition}

\theoremstyle{definition}

\theoremstyle{remark}
\newtheorem{rem}[thm]{Remark}

\numberwithin{equation}{section}

\DeclareMathSymbol{\C}{\mathalpha}{AMSb}{"43}

\newcommand{\bt}{\begin{thm}}
\newcommand{\et}{\end{thm}}
\newcommand{\beq}{\begin{equation}}
\newcommand{\eeq}{\end{equation}}

\newcommand{\eps}{\varepsilon}

\newcommand{\n}{\nabla}

\newcommand{\abs}[1]{\left|#1\right|}
\newcommand{\N}{\mathbb{N}}
\newcommand{\R}{{\mathbb{R}}}
\newcommand{\RT}{{\mathbb{R}^3}}

\newcommand{\E}{{\mathcal{E}}}

\newcommand{\K}{{\mathcal{K}}}
\newcommand{\A}{{\mathcal{A}}}
\renewcommand{\P}{{\mathcal{P}}}

\newcommand{\irt}{\int_{\mathbb{R}^3}}

\newcommand{\de}{\partial}

\newcommand{\bsub}{\begin{subequations}}
\newcommand{\esub}{\end{subequations}$\!$}

\begin{document}
	\title[Schr\"{o}dinger-Bopp-Podolsky system]{
On a nonlinear Schr\"{o}dinger-Bopp-Podolsky system\\ in the zero mass case: functional framework and existence}

\author[E. Caponio]{Erasmo Caponio}
\address{E. Caponio
\newline\indent Dipartimento di Meccanica, Matematica e Management,\newline\indent
	Politecnico di Bari
	\newline\indent
	Via Orabona 4,  70125  Bari, Italy}
\email{erasmo.caponio@poliba.it}
    
\author[P. d'Avenia]{Pietro d'Avenia}
\address{P. d'Avenia
\newline\indent Dipartimento di Meccanica, Matematica e Management,\newline\indent
	Politecnico di Bari
	\newline\indent
	Via Orabona 4,  70125  Bari, Italy}
\email{pietro.davenia@poliba.it}

\author[A. Pomponio]{Alessio Pomponio}
\address{A. Pomponio
\newline\indent Dipartimento di Meccanica, Matematica e Management,\newline \indent
	Politecnico di Bari
	\newline\indent
	Via Orabona 4,  70125  Bari, Italy}
\email{alessio.pomponio@poliba.it}

\author[G. Siciliano]{Gaetano Siciliano}
\address{G. Siciliano
\newline\indent Dipartimento di Matematica,\newline \indent
	Univerist\`a degli Studi di Bari Aldo Moro, 
	\newline\indent
	Via Orabona 4,  70125  Bari, Italy}
\email{gaetano.siciliano@uniba.it}

\author[L. Yang]{Lianfeng Yang}
\address{L. Yang
\newline \indent $^1$ School of Mathematics and Information Science, \newline \indent Guangxi University, 
\newline \indent Nanning, Guangxi, P. R. China
\newline\indent $^2$ Dipartimento di Meccanica, Matematica e Management,
\newline \indent Politecnico di Bari
\newline\indent	Via Orabona 4,  70125  Bari, Italy}
\email{yanglianfeng2021@163.com}

    \begin{abstract}
		In this paper, we consider in $\R^3$ the following zero mass Schr\"{o}dinger-Bopp-Podolsky system 
		\[
			\begin{cases}
				-\Delta u  +q^2\phi u=|u|^{p-2}u\\
				-\Delta \phi+a^2\Delta^2\phi=4\pi u^2
			\end{cases}
			\]
		where $a>0$, $q\ne 0$ and $p\in (3,6)$.
Inspired by \cite{R10}, we introduce a  Sobolev space $\E$ endowed with a norm containing a nonlocal term. Firstly, we provide some fundamental properties for the space $\E$ including embeddings into Lebesgue spaces. Moreover a general lower bound for the Bopp-Podolsky energy is obtained. Based on these facts, by applying a perturbation argument, we finally prove the existence of a weak solution to the above system.
\end{abstract}

	\maketitle

    \begin{center}
        \begin{minipage}{10cm}
	    \tableofcontents
	\end{minipage}
    \end{center}

\section{Introduction}
The Bopp-Podolsky theory, developed independently by Bopp \cite{B40} and Podolsky \cite{P42}, is a second-order gauge theory for the electromagnetic field that addresses the so-called {\sl infinity problem} associated with a point charge in the classical Maxwell theory. From a historical viewpoint, it is also natural to refer to the Bopp--Land\'e--Thomas--Podolsky theory, since related higher-order electrodynamic models were developed in the same period by Land\'e and Thomas; see, for instance, \cite{CarleyKiesslingPerlick2019} for a modern discussion and for an analysis of the Schr\"odinger spectrum of the hydrogen atom in this framework. Specifically, in Maxwell theory, by the well-known Gauss law (or Poisson equation), the electrostatic potential $\phi$ for a given charge distribution with density $\rho$ satisfies the equation
\begin{equation}\label{eq3.6}
	-\Delta \phi=\rho \qquad\text{ in } \R^3.
\end{equation}
If we let $\rho=4\pi \delta_{x_0}$ with $x_0\in \R^3$, then the fundamental solution of equation \eqref{eq3.6} is $\mathcal{C}(x-x_0)$ with $\mathcal{C}(x):=1/|x|$, and therefore the energy of the corresponding electrostatic field is not finite since
$$\int_{\R^3}|\nabla \mathcal{C}|^2dx=+\infty.$$
On the other hand, in the Bopp-Podolsky theory, equation \eqref{eq3.6} is replaced by
\begin{equation}\label{eq2.5}
	-\Delta \phi+a^2\Delta^2\phi=\rho \qquad\text{ in } \R^3
\end{equation}
with $a>0$.
Taking $\rho=4\pi \delta_{x_0}$ again, the explicit solution of equation \eqref{eq2.5} is $\mathcal{K}(x-x_0)$, where
\begin{equation*}
	\mathcal{K}(x):=\frac{1-e^{-\frac{|x|}{a}}}{|x|}
\end{equation*}
and, since
$$\int_{\R^3}|\nabla \mathcal{K}|^2dx+a^2\int_{\R^3}|\Delta \mathcal{K}|^2dx<+\infty$$
(see \eqref{nablaK} and \eqref{deltaK}), the energy of the electrostatic field generated by a point charge is finite.
     
  In addition, the Bopp-Podolsky theory may be interpreted as an effective theory at short distances,  while at  large distances it is experimentally indistinguishable from the Maxwell one. Both Bopp and Podolsky were influenced by the work of Yukawa \cite{yukawa} (see \cite[p. 345]{B40} and \cite[\S 4]{P42}) and indeed the {\em Bopp-Podolsky potential} $\K$ involves the Coulomb and the Yukawa potentials (see for more details Proposition~\ref{GreenDelta3}).

The coupling of quantum mechanics with Bopp-Podolsky electrodynamics leads naturally to  nonlinear systems that have attracted considerable attention in recent years.
	In \cite[Section 2]{DS19}, the authors coupled a Schr\"{o}dinger field $\psi=\psi(t,x)$ with its electromagnetic field in the Bopp-Podolsky
theory, and, in particular, in the electrostatic case for the standing waves $\psi(t,x)=e^{i\omega t}u(x)$, they studied the  following system 
	\begin{equation}\label{eq2.6}
		\begin{cases}
			-\Delta u +\omega u +q^2\phi u=|u|^{p-2}u , \\
			-\Delta \phi+a^2\Delta^2\phi=4\pi u^2,
		\end{cases}
        \text{ in } \R^3.
	\end{equation}
	
    From a physical point of view, in \eqref{eq2.6}, $a>0$, the parameter of the potential $\K$, has dimension of the inverse of mass in natural units and can be interpreted as a cut-off distance or can be linked to an effective radius for the electron; $q\neq0$ denotes the coupling constant of the interaction between the particle and its electromagnetic field; $\omega\in\R$ is the frequency of the standing wave $\psi(t,x)=e^{i\omega t}u(x)$;  $\phi$ is the electrostatic potential.
    
    Recently, a lot of works related to system  \eqref{eq2.6} have been carried out.  
    For instance, we refer the reader to the interesting results obtained in \cite{DPRS23,FS21,RSS24} for the problem with different constraints,  in \cite{CT20,LPT20,SDS24} for critical cases, in \cite{WCL22,Z24} concerning sign-changing solutions, in \cite{DG22,H19,H20} for nonlinear Schr\"odinger-Bopp-Podolsky-Proca system on  manifolds,  \cite{CLRT22,G23,S20} in further different contexts (see also references therein).

In the case $a=0$, system \eqref{eq2.6} reduces to the well known Schr\"odinger-Poisson system
	\begin{equation}\label{SP}
		\begin{cases}
			-\Delta u +\omega u +q^2\phi u=|u|^{p-2}u , \\
			-\Delta \phi=4\pi u^2,
		\end{cases}
        \text{ in } \RT,
	\end{equation}
which has been extensively studied by many authors.

It's well known that, for any fixed suitable $u$, the second equation of \eqref{SP} has a unique solution,
$1/|\cdot|*u^2$,
which vanishes at infinity, and therefore it is possible to rewrite \eqref{SP} as the nonlocal equation
\begin{equation}\label{SPeq}
    	-\Delta u +\omega u +q^2\left(\frac1{|x|}*u^2 \right)u=|u|^{p-2}u \qquad\text{ in } \RT.
\end{equation}

In particular, looking for static solutions, namely when $\omega=0$, the so called {\em zero mass} case has been considered in \cite{IR12,R10}. First,  for $p\in (18/7,3)$, Ruiz in \cite{R10} obtained the existence of a solution of \eqref{SPeq} as a minimizer of the associated energy functional 
\begin{equation*}
J(u) = \frac{1}{2}\|\nabla u\|_2^2 + \frac{q^2}{4}\irt \irt \frac{u^2(x) u^2(y)}{|x-y|}\, dx\,dy - \frac{1}{p}\|u\|^p_p
\end{equation*}
in the  functional space
\begin{equation*}
E_r=\left\{u\in D^{1,2}(\RT):u\text{ radial and }\irt \irt \frac{u^2(x) u^2(y)}{|x-y|}\, dx\,dy  <+\infty\right\},
\end{equation*}
where    $$D^{1,2}(\R^3):=\{u\in L^6(\R^3):|\nabla u|\in L^2(\R^3)\},$$ endowed with the usual norm $ \|\n \cdot\|_2$.
As an intermediate and crucial step  in his arguments, he established a general lower bound for the Coulomb energy which, in particular, applied to 
\[
D(u^2,u^2):=\irt \irt \frac{u^2(x) u^2(y)}{|x-y|}\, dx\,dy, 
\]
allowed him to get suitable embeddings of $E_r$ into Lebesgue spaces. In \cite{IR12}, instead, the authors proved the existence of a  ground state  and of infinitely many radial bound states for $p\in (3,6)$. For a related problem, see also \cite{MMV16,BGMMVS18}.

Motivated by \cite{IR12,MMV16,R10}, in this paper, we consider \eqref{eq2.6}  with $\omega=0$, namely we investigate
the zero mass Schr\"odinger-Bopp-Podolsky system
	\begin{equation}\tag{$\P$}\label{eq1.1}
		\begin{cases}
			-\Delta u +q^2\phi u=|u|^{p-2}u,  \\
			-\Delta \phi+a^2\Delta^2\phi=4\pi u^2,
		\end{cases} 
        \text{ in }  \R^3,
	\end{equation}
where $a>0$, $q\neq 0$, and $p\in (3,6)$.

	In order to state our main results, we introduce some notations and give some definitions. 
Let
\begin{equation*}
\mathcal{A}:=\Big \{\varphi\in  D^{1,2}(\R^3):\Delta\varphi \in L^2(\R^3)\Big\}
\end{equation*}
equipped with the scalar product
	$$\langle\varphi,\psi\rangle_{\mathcal{A}}:=\int _{\R^3}\nabla\varphi\cdot\nabla\psi dx+ a^2\int _{\R^3}\Delta\varphi\Delta\psi dx.$$
We denote by $\|\cdot\|_\A$ the associated norm.
From \cite{DS19}, we know that $C_c^\infty(\R^3)$  is dense in  the Hilbert space $\A$ and
	$$\mathcal{A}\hookrightarrow L^\tau(\R^3), \quad\text{for }\tau\in [6,+\infty].$$

For measurable functions $f,g\colon \R^3\to \R$,  we set
    \begin{equation}\label{eq:functional}
		V(f,g) :=  \int_{\R^3} \int_{\R^3} \K(x-y) f(x) g(y)  \,dx\,dy
	\end{equation}
and refer to it as {\em Bopp-Podolsky energy}, in analogy with potential theory (see \cite[Ch. 9]{LieLos01}).

    Finally, similarly to \cite{R10},
    we introduce the space $\E$ of the functions in $D^{1,2}(\R^3)$ 
    with finite  Bopp-Podolsky energy, i.e.
    \begin{equation}\label{E}
    \E:=\left\{u \in D^{1,2}(\RT) : V(u^2,u^2)<+\infty\right\}
    \end{equation}
    equipped with the norm
$$\|u\|_{\E}:=\left(\|\nabla u\|^2_2+V(u^2,u^2)^\frac{1}{2}\right)^\frac{1}{2}$$
(see Proposition \ref{PropA.2}).\\

In this paper we mainly investigate two types of results. The first type concerns  the characterization of the functional space $\mathcal E$ and some of its useful embedding properties.\\
Let us start with the following 

	\begin{thm}\label{thmE}
	We have
		\[
		\E=\{u\in D^{1,2}(\RT):\phi_u\in \A\},
		\]
		where 
		\begin{equation*}
			\phi_u(x):=(\K*u^2 )(x)=\int_{\R^3}\frac{1-e^{-\frac{|x-y|}{a}}}{|x-y|}u^2(y)dy.
\end{equation*}
Moreover, for each $u\in \E$,  $\phi_u$ is the unique weak solution in $\mathcal A$
		of $-\Delta \phi_u+a^2\Delta^2\phi_u=4\pi u^2$ in $\RT$,
        and
				\[
\|\phi_u\|^2_\mathcal A = 	\int_{\mathbb{R}^3} (|\nabla \phi_u|^2 + a^2 |\Delta \phi_u|^2) dx = 4\pi\int_{\R^3} \phi_u u^2 dx
=4\pi V(u^2,u^2).
\]
	\end{thm}
In addition, the following embedding result holds.
\begin{thm}\label{tE36}
The space $\E$ is continuously embedded into $L^\tau(\RT)$, for any $\tau\in [3,6]$. 
\end{thm}
Moreover, denoting by $\E_r$ the subspace of radial functions in $\E$, we can say something more on the embedding properties.
\begin{thm}\label{tEr187}
    The space $\E_r$ is continuously embedded into $L^\tau(\RT)$, for any $\tau\in (18/7,6]$. The embedding is compact for any $\tau\in (18/7,6)$.
\end{thm}
\begin{rem}
Since $E_r\subset\mathcal{E}_r$, by \cite[Theorem 1.2]{R10}, the embedding we found is optimal.
\end{rem}
A crucial step for the previous theorem is the following general lower bound for the Bopp-Podolsky energy.
    \begin{thm}\label{Thm1}
		Given $\alpha>1/2$, there exists a constant $C=C(a,\alpha)>0$ such that, for any measurable function $u:\R^3\to\R$ 
 there holds
		\begin{equation*}
			     V(u^2,u^2)\ge
   C \left(\int_{\R^3} \frac{u^2(x)}{(1+|x|^2)^{1/4}(1+|\log |x||)^\alpha} dx\right)^2.
		\end{equation*}
	\end{thm}
The above results are analogous to \cite{R10}, but 
while the properties of the Coulomb energy $D$ are well studied in the literature, still very few is known for the Bopp-Podolsky energy $V$ which creates several challenging differences. 
First of all, in \cite[Theorem 9.8]{LieLos01} it is proved that $D(f,f)$ is positive for any non-zero measurable function $f$ and this is a crucial step in \cite{R10} in order to construct a norm of the space $E$. The arguments in \cite{LieLos01} do not work for $V$ and a new proof is required. We also observe that it does not follow immediately from the positivity of the Fourier transform of $\mathcal{K}$ (see Remark \ref{kappahat} for details).
Moreover, due to the lack of homogeneity of $\K$, standard scaling arguments cannot be applied. Finally we remark that the inequality in Theorem~\ref{Thm1} is slightly different from the corresponding one in \cite[Theorem 1.1]{R10}. This is strictly related to the different behaviour of $\K$ and $\mathcal{C}$ at the origin and reflects the fact that the Bopp-Podolsky theory is  an effective theory at short distances,  while  it is indistinguishable from the Maxwell one at  large distances.

The second type of results deals with the existence of weak solutions to the Schr\"{o}dinger-Bopp-Podolsky system \eqref{eq1.1}, namely functions $u\in \E$ satisfying
		$$\int_{\R^3}\nabla u \cdot\nabla \varphi dx+q^2 \int_{\R^3}\phi_{u}u\varphi dx=\int_{\R^3}|u|^{p-2}u\varphi dx,\qquad \text{for all } \varphi\in\E.$$
	\begin{thm}\label{Thm2}
		Let $q\ne 0$. For any $p\in (3,6)$, there exists a non-trivial weak solution in $\E$ (or $\E_r$) to system \eqref{eq1.1}.  
		\end{thm}
For the proof we use an approximation procedure adding a small mass term, using some results contained in \cite{CLRT22,DS19} for the positive mass case, and passing to the limit. Some uniform estimates are necessary. Even if, in the radial setting, our space compactly embeds into $L^\tau(\RT)$, for any $\tau\in (18/7,6)$, due to the lack of homogeneity of our nonlocal term, we get the existence result only for $p\in(3,6)$. In contrast, in \cite{R10}, the case $p\in (18/7,3)$ is also considered but, therein, the scaling properties of the Coulombian energy nonlocal term play a crucial role.

This paper is organized as follows. In Section \ref{Se2} we establish some new results about the Bopp-Podolsky potential
equation, arriving to prove Theorem \ref{thmE}.
In Section \ref{Se3} we study the Bopp-Podolsky energy $V$ starting from preliminary results that allow us also to get Theorem~\ref{Thm1}.
In Section \ref{Se4} we analyse the functional space $(\E,\|\cdot\|_{\E})$ obtaining the embedding Theorems \ref{tE36} and \ref{tEr187}.
Finally, in Section \ref{Se5}, we focus on the regularity of the energy functional and we prove Theorem \ref{Thm2}.
	
We conclude with a list of notations:
	\begin{itemize}
		\item $\|\cdot\|_q$ denotes the usual $L^q(\R^3)$ norm for $1\le q\le +\infty$;
		\item let $B_{R}(x_0)$ denote the ball centered at $x_0$ with radius $R$; if $x_0=0$, we simply write $B_R$;
		\item $C$ is used to denote suitable positive constants whose value may be distinct from line to line;
        \item $o_n(1)$ indicates a vanishing sequence as $n\to+\infty$;
		\item the subscript $r$ denotes the respective subspace of radial functions.
	\end{itemize}
Moreover  all the considered functions are measurable.
\section{The electrostatic potential equation}\label{Se2}
We begin with a key proposition about the Bopp-Podolsky fundamental solution (see also \cite[Lemma 3.3]{DS19}).

Observe that, if $L$ is a linear operator and $u$ and $v$ satisfy in distributional sense $L(u)=\delta$ and $(L+k\operatorname{Id})(v)=\delta$, for $k\in\R\setminus\{0\}$, then
    \[
    \frac{1}{k} L\circ(L+k\operatorname{Id})[u-v]
    =
    \frac{1}{k} L(\delta + ku - \delta)
    =
    \delta.
    \]
In particular, in our case, by \cite[Theorems 6.20]{LieLos01} and \cite[Th. 6.23 with Remark (2)]{LieLos01}, we get
\begin{prop}\label{GreenDelta3}
		Let $x,y\in\R^3$ with $x\neq y$, $a>0$, and 
    \[
    \mathcal{C}_y(x)= \frac{1}{|x-y|}\quad
    \text{ and }\quad
    \mathcal{Y}_y(x)= \frac{e^{-\frac{|x-y|}{a}}}{|x-y|},
    \]
  be, respectively, the Coulomb and the Yukawa potentials in $\R^3$, namely, the solutions of 
$$-\Delta \mathcal C_y = 4\pi\delta_y, \ \ -\Delta \mathcal Y_y +\frac{1}{a^2}\mathcal Y_y = 4\pi\delta_y , \qquad\text{in }\RT,$$
 in distributional sense.
        Then 
    \[
    \mathcal{K}_y(x)=\mathcal{C}_y(x) - \mathcal{Y}_y(x)
=\frac{1-e^{-\frac{|x-y|}{a}}}{|x-y|}
    \]
  solves the equation $$-\Delta \K_y+a^2\Delta^2\K_y=4\pi \delta_y, \qquad \text{in } \RT,$$ in distributional sense.
	\end{prop}
\begin{rem}
Note that the above arguments establish a further relation between the Coulomb and Yukawa potentials. Indeed,  if we consider, for simplicity, the $\delta$
distribution centered in $0$, denoting $\mathcal{C}:=\mathcal{C}_0$ and $\mathcal{Y}:=\mathcal{Y}_0$,
then $\mathcal K =\mathcal C-\mathcal Y$ satisfies
$$-\Delta \mathcal K 
= \frac{1}{a^2}\mathcal Y \qquad\text{and}\qquad
-\Delta\mathcal K  +\frac{1}{a^2}\mathcal K 
= \frac{1}{a^2}\mathcal C$$
suggesting that, at least formally, respectively,
\begin{equation*}
   \mathcal K= \frac{1}{a^2}\mathcal{C}*\mathcal{Y}
\qquad\text{and}\qquad
\mathcal K=\frac{1}{a^2}\mathcal{Y}*\mathcal{C} 
\end{equation*}
which make sense due to the commutativity of the convolution.
This can be obtained rigorously by explicitly computing the convolution. In fact, passing 
in spherical coordinates and choosing the $z$-axis with the same direction of $x\neq0$, 
$y=(r, \theta,\vartheta)$
with $r\ge 0, \theta\in[0,\pi], \vartheta\in[0,2\pi]$ and with $R=|x|$,
we have
\begin{eqnarray*}
     \int_{\mathbb R^3} \frac{e^{-|y|/a}}{|x-y| |y|}dy &=&
    \int_{0}^{+\infty} \int_{0}^{\pi} \int_{0}^{2\pi}
    \frac{e^{-r/a}}{\sqrt{R^2+r^2-2Rr \cos\theta}}r\sin\theta d\vartheta d\theta  dr\\
    &=& 2\pi\int_{0}^{+\infty}\int_{0}^{\pi}\frac{e^{-r/a}}{\sqrt{R^2+r^2-2Rr \cos\theta}}r\sin\theta d\theta dr\\
    &=&2\pi\int_{0}^{+\infty}\int_{(r-R)^2}^{(r+R)^2}
    \frac{e^{-r/a}}{s^{1/2}} \frac{ds}{2R}dr\\
    &=&\frac{\pi}{R}\int_{0}^{+\infty} e^{-r/a} 2[(r+R) - |r-R|] dr\\
    &=& \frac{4\pi}{R}\int_{0}^{R}e^{-r/a}r dr + 4\pi\int_{R}^{+\infty}e^{-r/a}dr= 4\pi a^2 \frac{1-e^{-R/a}}{R},
\end{eqnarray*}
where we used the change of variable $s=R^2+r^2-2Rr\cos\theta$.
\end{rem}
Since the arguments below do not depend on the particular value of $a>0$, for notational simplicity we set $a=1$.

Proposition~\ref{GreenDelta3} allows us to represent solutions of the equation  
        \begin{equation}\label{BPeq}
			-\Delta v+\Delta^2 v=4\pi f\qquad \text{in $\RT$} 
		\end{equation} 
for a suitable function $f$.
We say that $v$ satisfies \eqref{BPeq} in weak sense if, for any $\varphi\in C_c^\infty(\RT)$ we have
\[
\irt (\n v \cdot \n \varphi+\Delta v \Delta\varphi)dx=4\pi \irt f \varphi dx,
\]
whenever this is meaningful. Moreover we set
\[\psi_f := \K *f.\]
	\begin{thm}\label{solu}
		Let $f\geq 0$. If $\psi_f\in L^1_{\mathrm{loc}}(\R^3)$ then it solves the equation \eqref{BPeq}	
		in distributional sense.
	\end{thm}
 \begin{rem}\label{remL1loc}
The above theorem can be extended  to any differential operator having a fundamental solution satisfying, with $f\geq 0$, the  integrability assumption there. 
Moreover, the proof of the theorem makes clear that $f$  defines ({\em a posteriori}) a distribution. Anyway, for our $\K$, it is clear that if $\psi_f(x)\in\R$, for at least one $x\in \RT$, then $f\geq 0$ is locally integrable. In fact, being
$\K\in C(\RT\setminus\{0\})$, continuously extendible  in $0$, and positive, if $K$ is a compact set in $\RT$, then 
\[
+\infty>\psi_f(x)=\int_K \K(x-y)f(y)dy \geq C\int_K f(y) dy.
\]
\end{rem}   
In order to prove Theorem~\ref{solu} we need the following   
\begin{lem}\label{key}
Let $h$ be a function such that  $\K*|h|$ be locally integrable and  $g\in L^\infty(\RT)$ with compact support. Then 
\[
	\int_{\RT}(\K*h)  g  dx =\int_{\RT} (\K*g)  h dx.
\]
\end{lem}	
\begin{proof}
Since 
\[
	\int_{\RT} \left( \int_{\RT} \K(x - y) |h(y)| \, dy \right)| g(x) |\, dx<+\infty, 
\]
and $\K>0$, by Fubini-Tonelli theorem we get
	\begin{align*}
	\int_{\RT}(\K*h)  g  dx&=\int_{\RT} \left( \int_{\RT} \K(x - y) h(y) \, dy \right) g(x)  dx\\
	&	 = \int_{\RT} h(y) \left( \int_{\RT} \K(x - y) g(x)  dx \right) dy\\
	&= \int_{\RT} h(y) \left( \int_{\RT}  \K(y-x)g(x)  dx \right) dy= \int_{\RT} (\K*g)  h dx
\end{align*}
where we have used $\K(-x)=\K(x)$.
\end{proof}
Thanks to the above lemma we have
\begin{proof}[Proof of Theorem~\ref{solu}]
By Lemma~\ref{key} and Proposition~\ref{GreenDelta3}, for any test function $\varphi \in C_c^\infty(\RT)$ we get
		\begin{align*}
			-\int_{\RT} \psi_f\Delta \varphi dx 
            +\int_{\RT} \psi_f\Delta^2 \varphi dx
            &= -\int_{\RT}
            (\mathcal{K}*f)\Delta \varphi dx +
            \int_{\RT}
            (\mathcal{K}*f)\Delta^2 \varphi dx 
            \\
            &= \int_{\RT} f \cdot \K * (-\Delta \varphi) dx
            + \int_{\RT} f \cdot \K * (\Delta^2 \varphi) dx\\
            &=\int_{\RT} f \cdot \Big(\int_{\RT} \K(x-y)  \big(-\Delta \varphi (y)+\Delta^2 \varphi(y)\big) dy \Big)dx\\
            &=4\pi \int_{\RT} f \varphi dx.
		\end{align*}
    \end{proof}
Let us also define the following spaces
\[\begin{aligned}
	\mathcal E_1^+ &:= \left\{f:\R^3\to \R: f\geq 0\text{ and }V(f,f)< +\infty\right\},\\
	\mathcal E_2^+ &:= \left\{f: \R^3\to \R :  f\geq 0\text{ and }\psi_f \in \mathcal A\right\}.
\end{aligned}\]
Theorem \ref{thmE} is an immediate consequence of the following result.
\bt\label{mainplus}
The spaces $\mathcal E_1^+$ and $\mathcal E_2^+$ coincide. 
Moreover, for any $f \in \mathcal E_2^+$, $\psi_f$  satisfies \eqref{BPeq}
in weak sense  and the following energy identity holds
\beq\label{energyidentity}\int_{\mathbb{R}^3} (|\nabla \psi_f|^2 + |\Delta \psi_f|^2) dx = 4\pi \int_{\R^3} \psi_f f dx.\eeq
\et
As an intermediate step toward the proof of Theorem~\ref{mainplus}, let us consider the case where $f\in L^{6/5}(\R^3)$ with compact support.
\begin{prop}\label{bounded}
	Let $f \in L^{6/5}(\R^3)$ with compact support. Then $\psi_f$
	belongs to $\mathcal A$, solves \eqref{BPeq} in weak sense, and satisfies \eqref{energyidentity}.
	\end{prop}
\begin{proof}
	Let $\rho \in C_c^\infty(\R^3)$ be a standard mollifier, with support in
    $B_1$, and set $\rho_n(x) = n^3\rho(nx)$, with $n\in \N$. Define
	\[f_n := f * \rho_n.\]
	Then, for every $n\in\N$, $f_n \in C_c^\infty(\R^3)$,  there exists $b>0$ such that $\bigcup_n \operatorname{supp}(f_n)\subset B_b$,
    and $f_n \to f$ in $L^\tau(\RT)$ for all $\tau \in [1,6/5]$, as $n\to +\infty$.
    Let 	
    \[\psi_n: = \psi_{f_n}=\K * f_n.\]
	Note  that, for each $n\in \N$, $\psi_n \in C^\infty(\mathbb{R}^3)$ and, being $\|\K\|_\infty\leq 1$, $\|\psi_n\|_\infty\leq \|f_n\|_1$. Moreover, as a consequence of Proposition \ref{GreenDelta3}, $\psi_n$ satisfies
	\[-\Delta \psi_n +\Delta^2\psi_n= 4\pi f_n\]
	in distributional sense and, actually,  classically.
	\\
	For any $R > 0$, 
    by Green's formula, we have
 \[\int_{B_R} |\nabla \psi_n|^2 dx = -\int_{B_R} \psi_n \Delta \psi_n dx + \int_{\partial B_R} \psi_n \partial_\nu \psi_n dS,\]
	where
	\[\partial_\nu \psi_n(x) = \int_{\R^3} \de_\nu\K(x-y) f_n(y)  dy.\]
	\\
If $R>b$, $x \in \partial B_R$, and $y \in \mathrm{supp}(f_n)$, we have $|x-y| \geq R - b$.
	Therefore,		for all $x\in \partial B_R$,
    \begin{equation}\label{psidecay}
		|\psi_n(x)| \leq \frac{1}{R-b} \|f_n\|_1.
	\end{equation}
	Moreover, since
	\beq\label{nablaK}
	\nabla \K(x)=\frac{x}{ |x|^3}\Big(e^{-|x|}(1+|x|)-1\Big)\qquad
    \text{ in } \R^3\setminus\{0\}
	\eeq
    and $-1\leq e^{-|x|}(1+|x|)-1\leq 0$, then, for all $x\in \partial B_R$,
		\[|\partial_\nu \psi_n(x)| \leq \frac{1}{(R-b)^2} \|f_n\|_1.\]
Thus,
	\[\left|\int_{\partial B_R} \psi_n \partial_\nu \psi_n dS\right| \leq \frac{4\pi R^2}{(R-b)^3} \|f_n\|_1^2\]
	which  vanishes as $R \to +\infty$.
	Moreover, by applying  the Green's formula, we get 		
	\beq\label{green2}\int_{B_R} |\Delta \psi_n|^2 dx = \int_{B_R} \psi_n \Delta^2 \psi_n dx - \int_{\partial B_R} \psi_n \partial_\nu \Delta \psi_n dS
	+ \int_{\partial B_R} \Delta \psi_n \partial_\nu \psi_n dS.\eeq	
	Since,
    \begin{equation}
        \label{deltaK}
    \Delta \K(x)=-\frac{e^{-|x|}}{|x|}
    \quad \text{and}\quad
 \nabla \Delta \K(x)=\frac{x}{|x|^3}e^{-|x|}(1+|x|),\qquad\text{ in } \R^3\setminus\{0\},
    \end{equation}
repeating the  reasoning  above,  also the boundary terms in \eqref{green2} go to zero as $R\to +\infty$.
	Therefore, we obtain
	\begin{equation}
		\label{psii}\int_{\R^3} \left(|\nabla \psi_n|^2+|\Delta \psi_n|^2\right) dx
        = \int_{\R^3} \psi_n\left(-\Delta\psi_n+\Delta^2 \psi_n\right) dx
        = 4\pi \int_{\R^3} \psi_n f_n dx\leq  4\pi \|f_n\|_1^2.
	\end{equation}
	From \eqref{psii},  $\nabla\psi_n\in L^2(\R^3)$ and $\Delta\psi_n\in L^2(\R^3)$ and then, taking into account  \eqref{psidecay},  
$\psi_n\in \mathcal A$.
	Since the right-hand side in \eqref{psii} is uniformly bounded, we deduce that $\{\psi_n\}$ is a bounded sequence in $\mathcal A$ and then, being $\A$  an Hilbert space, up to a subsequence, there exists $\xi\in \A$ such that
    $\psi_n \rightharpoonup \xi$ weakly in $\A$,
    as $n\to +\infty$. Moreover, since
\[
	\|\psi_n-\psi_f\|_\infty=\|\K*(f_n-f)\|_\infty  \le \|f_n-f\|_{1}\to 0, \qquad\text{as }n\to +\infty,
	\]
	then, up to a subsequence, $\psi_n\to \psi_f$ a.e. in $\RT$ and so $\psi_f=\xi\in\A$.
	\\
	Therefore, for any test function $\varphi \in C_c^\infty(\R^3)$, we have 
	\begin{equation*}
\begin{split}
    		\int_{\R^3} \left(\nabla \psi_f \cdot \nabla \varphi +\Delta \psi_f  \Delta \varphi\right) dx
		&= \lim_{n \to +\infty} \int_{\R^3} \left(\nabla \psi_n \cdot \nabla \varphi +\Delta \psi_n  \Delta \varphi\right) dx
        \\
        &= \lim_{n \to +\infty} 4\pi \int_{\R^3}  f_n \varphi dx= 4\pi \int_{\R^3}  f \varphi dx.
\end{split}
	\end{equation*}
	This shows that $\psi_f$ satisfies $-\Delta \psi+\Delta^2\psi=4\pi f$ in weak sense and, since $f\in L^{6/5}(\R^3)$, we get \eqref{energyidentity} by density of test functions in $\A$.
\end{proof}
\begin{lem}\label{l1loc}
If $f\ge 0$ is such that 
$V(f,f)<+\infty$, then $f\in L^{1}_{\mathrm{loc}}(\R^3)$.
\end{lem}
\begin{proof}
Fix an arbitrary $K\subset \R^3$ compact, since 
    \[V(f,f)\geq\int_K\int_K \K(x-y)f(x)f(y)dx dy \geq 
C\left(\int_K f(x) dx\right)^2,\]
    we conclude that $f\in L^{1}_{\mathrm{loc}}(\R^3)$.
\end{proof}
We are ready to prove Theorem~\ref{mainplus}.
\begin{proof}[Proof of Theorem~\ref{mainplus}]
	Let us show first that $\mathcal E_1^+\subset \mathcal E_2^+$ and so take $f \in \E_1^+$. By Lemma \ref{l1loc}   we have that $f\in L^{1}_{\mathrm{loc}}(\R^3)$. Let us now consider the  approximating sequence
	$f_n := \chi_n  \min\{n,f\}$, where 	$\chi_n$ is the characteristic function of
    $B_n$. Hence, for any $n\ge 1$, we have that
	$0\leq f_n \leq n$, $f_n$ has compact support in $B_n$, and 
	$f_n\to f$ a.e. on $\R^3$. 
	Let    $\psi_n := \psi_{f_n}=\K * f_n$. 
By Proposition  \ref{bounded}, each $\psi_n$ satisfies in weak sense
\begin{equation}\label{eq4.1}
	-\Delta \psi_n + \Delta^2\psi_n = 4\pi f_n, \qquad\text{in }\RT,
    \end{equation}
	and
	\begin{equation}
	    \label{4.1.5}
        \int_{\R^3} (|\nabla \psi_n|^2 + |\Delta \psi_n|^2) dx =4\pi \int_{\R^3}\psi_n f_n dx. %= 
	\end{equation}
Since, for any $n\ge 1$, $0\le f_n\leq f_{n+1}\leq f$, and so $0\le \psi_n\le \psi_{n+1}\le \psi_f$,
    being $f\in \mathcal{E}_1^+$, we have 
    \begin{equation}\label{eq4.2}
	0\le\int_{\R^3}\psi_n f_n \leq \int_{\R^3}\psi_f f dx = V(f,f)<+\infty.
    \end{equation}
Then, by \eqref{4.1.5} $\{\psi_n\}$ is bounded in $\mathcal{A}$ and so we can extract a weakly convergent subsequence $\psi_n \rightharpoonup \xi$ in $\A$.
By taking the limit in the weak formulation for \eqref{eq4.1}, we  get   for any test function $\varphi \in C_c^\infty(\R^3)$
	\begin{equation*}
		\int_{\R^3} (\nabla \xi \cdot \nabla \varphi +\Delta \xi\Delta \varphi) dx =4\pi \int_{\R^3}  f \varphi dx.
	\end{equation*}
	This shows that $\xi$ satisfies \eqref{BPeq} in weak sense.
On the other hand 
	from the monotone convergence theorem 
    $\psi_n\to \psi_f$ a.e. in $\R^3$. Hence $\psi_f=\xi$ a.e. in $\R^3$,  therefore $\psi_f\in\A$ and so $f \in \E_2^+$.
	
To show that $\E_2^+ \subset \E_1^+$, let $f \in \E_2^+$. As $\psi_f\in \A$ then $\psi_f\in L^1_{\mathrm{loc}}(\R^3)$ and so, by Theorem~\ref{solu} (recall also Remark~\ref{remL1loc}),   
we get that $\psi_f$ satisfies \eqref{BPeq} in distributional sense  and, furthermore, for any $\varphi\in C_c^\infty(\mathbb{R}^3)$, we have
\[
 \int_{\mathbb{R}^3} (\nabla \psi_f \cdot \nabla \varphi +\Delta\psi_f \Delta \varphi)dx = 4\pi \int_{\mathbb{R}^3} f\varphi dx. 
\]
As $\psi_f\geq 0$,  arguing as in the proof of  \cite[Lemma 3.2]{DS19} there exists a sequence $\{\varphi_n\}\subset C_c^\infty(\mathbb{R}^3)$ such that $\varphi_n \geq 0$ for all $n$
	and $\varphi_n \to \psi_f$ in $\A$.
	Therefore, up to pass to a subsequence, $\varphi_n\to \psi_f$ a.e. and,  by Fatou's lemma, 
    \begin{equation}
        \label{end1}
        \begin{split}
           4\pi V(f,f)
           &=
           4\pi\int_{\mathbb{R}^3} f \psi_f dx
           \leq
           4\pi \lim_{n \to +\infty} \int_{\mathbb{R}^3} f\varphi_n dx\\
           &= \lim_{n \to +\infty} \int_{\mathbb{R}^3} (\nabla \psi_f \cdot \nabla \varphi_n +\Delta\psi_f \Delta \varphi_n)dx
           =           \int_{\mathbb{R}^3} (|\nabla \psi_f|^2+|\Delta\psi_f|^2) dx, 
        \end{split}
    \end{equation}
	i.e. $f\in \E_1^+$. 
    
To conclude, we observe that we showed above that, if $f\in \mathcal{E}_1^+$, then
$\psi_f$ solves \eqref{BPeq} in weak sense. Moreover, to get \eqref{energyidentity}, we can use again a sequence $\{f_n\}$ and the corresponding $\{\psi_n\}$ as above. By using the  weak lower semi-continuity of the norm and \eqref{eq4.2}, we derive that
	\[
	\int_{\mathbb{R}^3} (|\nabla \psi_f|^2 +|\Delta\psi_f|^2) dx \leq  \lim_{n \to +\infty} \int_{\mathbb{R}^3} (|\nabla \psi_n|^2 +|\Delta\psi_n|^2) dx = 4\pi\lim_{n \to +\infty} \int_{\R^3}\psi_n  f_n dx\le  4\pi\int_{\mathbb{R}^3} \psi_f f dx, \]
	which together with \eqref{end1} completes the proof.
\end{proof}
\section{Properties of the Bopp-Podolsky energy}\label{Se3}
In this section we focus on 
some properties of the Bopp-Podolsky energy $V$ defined in \eqref{eq:functional}. 
We notice the following fact.
    \begin{lem}
	For all $x\in\RT$, it holds
		\begin{equation*}
			1-e^{-|x|}= |x|\int_0^1 e^{-t|x|} \,dt,
		\end{equation*}
        and so, for any $f,g$
        such that $V(|f|,|g|)<+\infty$,
        it results
		\begin{equation*}
		V(f,g) =\int_0^1\int_{\R^3} \int_{\R^3} e^{-t|x-y|}f(x)g(y) \,dx\,dy\,dt.
	\end{equation*}
	\end{lem}
	Next, we are going to show the following	\begin{lem}\label{lem:main}
 		 For any non-zero function $f \in L^2(\R^3)$ such that $V(|f|,|f|)<+\infty$, we have that $V(f,f)$ is strictly positive.
	\end{lem}		
	\begin{proof} 
For  $t>0$, and $k\in \R^3$, let $g(t,k):=\mathcal{F}[e^{-t|\cdot|}](k)$, where $\mathcal{F}$ denotes the Fourier transform. We know that $g(t,k)>0$, for all $t>0$, and $k\in \R^3$.
For any $t>0$, by the convolution theorem and the 
Parseval's formula, we get
		\begin{equation*}
			\int_{\R^3} \int_{\R^3} e^{-t|x-y|}f(x)f(y) \,dx\,dy 
            =            \int_{\R^3} \abs{\mathcal{F}[f](k)}^2 g(t,k)\,dk >0.
		\end{equation*}
By using Fubini's theorem, we have that
		\[V(f,f)=\int_0^1\int_{\R^3} \int_{\R^3} e^{-t|x-y|} f(x)f(y)\,dx\,dy\,dt>0\]
concluding the proof.
	\end{proof}
	Furthermore, we have
    
\begin{prop}\label{prop:nonneg}
		Let $f:\R^3\to \mathbb R$ satisfy $V(|f|, |f|)<+\infty$. Then
		$$V(f,f)\geq0.$$
	\end{prop}
	\begin{proof}
We can take a non-zero function  $f$ and consider  the sequence
        \begin{equation*}
			f_n := \chi_n \min\{n, \max\{-n, f\}\}, 
		\end{equation*}
		where 	$\chi_n$ is the characteristic function of $B_n$. By Lemma \ref{lem:main}, we obtain $V(f_n, f_n)>0$, for each $n\ge 1$. Moreover,
        since $f_n\to f$ a.e. in $\R^3$,
        by Lebesgue's dominated convergence theorem, $V(f_n, f_n)\to V(f,f)$,  as $n\to +\infty$, and we are done.	    
	\end{proof}
We have not yet proved that $V$ is a well defined bilinear form when we consider functions $f$ such that $V(|f|,|f|)<+\infty$. This is shown in the next proposition together with a 
Cauchy–Schwarz type inequality.
	\begin{prop}\label{prop:main}
		For all  functions $f,g:\RT\to\R$, with $V(|f|, |f|), V(|g|, |g|)<+\infty$, we have that
		\beq\label{CS}|V(f,g)|^2\leq V(f,f)V(g,g).\eeq
	\end{prop}	
	\begin{proof}
Let us first show  that $V(f,g)\in\R$. 
Consider $V(|f_n|, |g_n|)$,  where $$f_n := \chi_n f,\qquad g_n := \chi_n g$$
and $\chi_n$ is the characteristic function of $B_n$.
Observe that, Lemma \ref{l1loc} implies that $f,g\in L_{\rm loc}^1(\R^3)$, then
$$V(|f_n|, |g_n|)
\leq
\left(\int_{B_n} |f|dx\right)\left(\int_{B_n} |g|dx\right)<+\infty$$
and, moreover, $V\left(\big|\alpha|f_n|+\beta|g_n|\big|,\big|\alpha|f_n|+\beta|g_n|\big|\right)<+\infty$, for any $\alpha, \beta\in\R$. So,
being $\K$ an even positive function, for any $\alpha, \beta\in\R$, 
$$0\leq V(\alpha |f_n|+\beta |g_n|, \alpha |f_n|+\beta |g_n|)=\alpha^2V(|f_n|,|f_n|)+2\alpha\beta V(|f_n|,|g_n|)+\beta^2 V(|g_n|,|g_n|).$$ 
Then 
\[V(|f_n|, |g_n|)^2\leq V(|f_n|, |f_n|)V(|g_n|, |g_n|)\]
and so, by Fatou's lemma and Lebesgue's dominated convergence theorem, we get 
\[V(|f|, |g|)^2\leq V(|f|, |f|)V(|g|, |g|)<+\infty.\]
As $|V(f,g)|\leq V(|f|,|g|)$, we obtain $V(f,g)\in\R$.
\\
Thus, for any $\alpha, \beta\in\R$, 
being 
\begin{equation}\label{vecspace}
V(|\alpha f+\beta g|, |\alpha f+\beta g|)\leq \alpha^2V(|f|,|f|)+2|\alpha\beta| V(|f|,|g|)+\beta^2 V(|g|,|g|)<+\infty,
\end{equation}
by Proposition \ref{prop:nonneg} 
$$0\leq V(\alpha f+\beta g, \alpha f+\beta g)=\alpha^2V(f,f)+2\alpha\beta V(f,g)+\beta^2 V(g,g),$$
and then we get \eqref{CS}.
\end{proof}
The previous results allow us to extend Theorem \ref{mainplus} removing the non-negativity assumptions. Indeed, let $f^+:=\max\{f,0\}$, $f^-:=\max\{-f,0\}$, and
\begin{align*}
	\mathcal E_1&:= \left\{f:\RT\to\R:\  V(|f|,|f|) < +\infty\right\},\\
	\mathcal E_2 &:= \left\{f:\RT\to\R:\  \psi_{f^+}, \psi_{f^-} \in \mathcal A\right\}.\nonumber
\end{align*}
We have
\bt\label{main}
The spaces $\mathcal{E}_1$ and $\mathcal{E}_2$ coincide. Moreover, if $f\in \E_2$, then $\psi_f\in \A$,  it satisfies \eqref{BPeq} in weak sense, and \eqref{energyidentity} holds.
\et	
\begin{proof}
First let us prove that $\mathcal E_1= \E_2$. Due to the symmetry of the kernel, we have, formally,
\[
	V(|f|,|f|)
    =V(f^+ +f^-, f^+ +f^-)
    =V(f^+, f^+)+V(f^-, f^-)+2V(f^+,f^-).
\]
Hence, using Proposition \ref{prop:main},  $V(|f|,|f|)$ is finite if and only if both $V(f^+, f^+)$ and $V(f^-, f^-)$ are finite.\\
Thus, if $f\in\E_1$ then $f^+,f^-\in \E_1^+$ and, from Theorem~\ref{mainplus}, $\psi_{f^+},   \psi_{f^-}\in \A$.  Vice versa, if $f\in \E_2$ then, 
again by Theorem~\ref{mainplus}, $f^+, f^-\in \E_1^+$ and then $f\in \E_1$.

Now, let $f\in \E_2$. Observe that, by the linearity of the convolution, we have $\psi_f = \psi_{f^+} -  \psi_{f^-}$. Hence, by Theorem \ref{mainplus}, 
we get that $\psi_f\in\A$ and  $\psi_f$ solves \eqref{BPeq} in 
weak sense.

Finally, let us show that the energy identity  \eqref{energyidentity} is satisfied as well. By Theorem~\ref{mainplus} we have
\[	\begin{aligned}
		&\int_{\R^3} |\nabla \psi_{f^+}|^2dx + \int_{\R^3} |\Delta \psi_{f^+}|^2dx = 4\pi \int_{\R^3}\psi_{f^+}f^+dx, \\
		&\int_{\R^3} |\nabla \psi_{f^-}|^2dx + \int_{\R^3} |\Delta \psi_{f^-}|^2 dx= 4\pi\int_{\R^3} \psi_{f^-}f^-dx.
	\end{aligned}\]
    Then
    \begin{align*}
		\int_{\R^3} |\nabla \psi_f|^2dx + \int_{\R^3} |\Delta \psi_f|^2dx
        &= \int_{\R^3} |\nabla(\psi_{f^+} - \psi_{f^-})|^2dx + \int_{\R^3} |\Delta(\psi_{f^+} - \psi_{f^-})|^2 dx
        \\
	&	= \int_{\R^3} |\nabla \psi_{f^+}|^2dx - 2\int_{\R^3} \nabla \psi_{f^+} \cdot \nabla \psi_{f^-} dx+ \int_{\R^3} |\nabla \psi_{f^-}|^2 dx
        \\ 
        &\qquad+ \int_{\R^3} |\Delta \psi_{f^+}|^2dx- 2\int_{\R^3} \Delta \psi_{f^+}  \Delta \psi_{f^-} dx+ \int_{\R^3} |\Delta \psi_{f^-}|^2dx\\
        &=4\pi V(f^+,f^+) + 4\pi V(f^-,f^-)\\
		&\quad- 2\int_{\R^3} \nabla \psi_{f^+} \cdot \nabla \psi_{f^-}dx - 2\int_{\R^3} \Delta \psi_{f^+}  \Delta \psi_{f^-}dx.
	\end{align*}
Taking into account that
		\[V(f,f)=V(f^+-f^-, f^+-f^-)=V(f^+, f^+)+V(f^-, f^-)-2V(f^+,f^-),\]
	for the energy identity to hold, we need to observe that
	\begin{equation*}
		\int_{\R^3} \nabla \psi_{f^+} \cdot \nabla \psi_{f^-} dx+ \int_{\R^3} \Delta \psi_{f^+}  \Delta \psi_{f^-} dx= 4\pi V(f^+,f^-),
	\end{equation*}
    which can be obtained by Theorem \ref{mainplus} taking, for instance, $\psi_{f^+}$ as weak solution and $\psi_{f^-}$ as test function.
\end{proof}
As a consequence of the previous result we obtain that $V$ is a positive definite bilinear form on $\mathcal{E}_1$.
\begin{prop}\label{Vpositive}
If $f\in \E_1\setminus\{0\}$,  then  $V(f,f)>0$. 
\end{prop}
\begin{proof}
By Proposition~\ref{prop:nonneg}, we know that $V(f,f)\geq 0$. By contradiction, let us assume that  $V(f,f)=0$. By \eqref{energyidentity}, we then get
$		0=4\pi V(f,f) 
=\|\psi_f\|_\A^2$, 
which implies that 
$\psi_f=0$. Since $\psi_f$ satisfies \eqref{BPeq} by Theorem~\ref{main}, we conclude that $f=0$.
\end{proof}
\begin{rem}\label{preHilbert}
		If $f,g\in \E_1$, \eqref{vecspace} shows that $\E_1$ is a vector space on $\R$. Moreover, from Propositions~\ref{prop:nonneg} and \ref{Vpositive}, $V$ is a inner product on $\mathcal E_1$.
	\end{rem}
\begin{rem}\label{kappahat}
    Observe that, since $\mathcal{F}(\K)(\xi)=(|\xi|^2+|\xi|^4)^{-1}>0$, for any $\xi\neq 0$, if $f\in L^2(\R^3)\setminus\{0\}$ and we know that $\K\ast f\in L^2(\R^3)$, we have 
    \[
    V(f,f)=\irt \mathcal{F}(\K)(\xi)|\hat{f}(\xi)|^2d \xi>0.
    \]
    However, in general, we cannot deduce that $\K\ast f\in L^2(\RT)$, even if $V(|f|,|f|)<+\infty$. For example, it's easy to check that if $f$ is the characteristic function of the ball $B_1$, we have that $V(|f|,|f|)<+\infty$ but $\K\ast f\notin L^2(\RT)$.
\end{rem}
 Finally, inspired by \cite[Theorem 3.1]{R10}, we establish the following interesting inequality, that implies Theorem~\ref{Thm1}.
\begin{thm}\label{thm3.1}
		Given $\alpha>1/2$, there exists some constant $C=C(\alpha)>0$ such that, for any  function $f:\R^3\to[0,+\infty)$, there holds that
		\begin{equation*}
			    V(f,f)\ge
   C \left(\int_{\R^3} \frac{f(x)}{(1+|x|^2)^{1/4}(1+|\log |x||)^\alpha} dx\right)^2.
		\end{equation*}
	\end{thm}
	\begin{proof}
Observe that
    \begin{align*}
        V(f,f)
        &=
        \int_{\R^3}dy\int_{\R^3}\mathcal{K}(x-y) f(x) f(y) dx\\
        &\geq
        \int_{\R^3}dy\int_{|y|<2|x|<4|y|} \mathcal{K}(x-y) f(x) f(y) dx\\
        &=
        \int_{B_1}dy\int_{|y|<2|x|<4|y|} \mathcal{K}(x-y) f(x) f(y) dx
        +
        \int_{B_1^c}dy\int_{|y|<2|x|<4|y|}\mathcal{K}(x-y) f(x) f(y) dx.
\end{align*}

If $|y|<1$ and $|y|<2|x|<4|y|$, then $|x-y|\leq |x|+|y|<2|y|+1<3$ and so $\mathcal{K}(x-y)\geq C>0$. Thus
        \[
        \int_{B_1}dy\int_{|y|<2|x|<4|y|} \mathcal{K}(x-y) f(x) f(y) dx
        \geq
        C \int_{B_1}dy\int_{|y|<2|x|<4|y|} f(x) f(y) dx.
        \]

If $|y|\geq 1$ and $|y|<2|x|<4|y|$, so that $|x|>1/2$, we distinguish two cases.\\
If, in addition, $|x-y|\leq 1$, then
\[
\mathcal{K}(x-y) |x|^\frac{1}{2} |y|^\frac{1}{2}
\geq C.
\]
If, in addition, $|x-y|\geq 1$, then
\[
\mathcal{K}(x-y) |x|^\frac{1}{2} |y|^\frac{1}{2}
=
(1-e^{-|x-y|})\frac{|x|^\frac{1}{2} |y|^\frac{1}{2}}{|x-y|}
\geq C.
\]
Thus
        \[
        \int_{B_1^c}dy\int_{|y|<2|x|<4|y|}\mathcal{K}(x-y) f(x) f(y) dx
        \geq
        C\int_{B_1^c}dy\int_{|y|<2|x|<4|y|}\frac{f(x)}{|x|^\frac{1}{2}} \frac{f(y)}{|y|^\frac{1}{2}} dx.
        \]
        Hence, if
        \[
        \mathcal{J}(r)=
        \begin{cases}
            1 & \text{ for }0\leq r \leq 1/2 \\
            1/\sqrt{2r} & \text{ for } r > 1/2,
        \end{cases}
        \]
        we have that
        \begin{align*}
            V(f,f)
        &\geq
        C \left(\int_{B_1}dy\int_{|y|<2|x|<4|y|} f(x) f(y) dx
        +
        \int_{B_1^c}dy\int_{|y|<2|x|<4|y|}\frac{f(x)}{|x|^\frac{1}{2}} \frac{f(y)}{|y|^\frac{1}{2}} dx\right)\\
        &\ge 
        C \int_{\R^3}dy\int_{|y|<2|x|<4|y|} \mathcal{J}(|x|)f(x) \mathcal{J}(|y|)f(y) dx\\
        &= C
        \int_0^{+\infty}dr\int_{r/2}^{2r} \left(\mathcal{J}(r)\int_{|y|=r} f(y) d\sigma_y\right) \left(\mathcal{J}(s)\int_{| x|=s} f(x) d\sigma_x\right) ds.
        \end{align*}
Now, taking
        \[
        h(r)=\frac{\mathcal{J}(r)}{(1+|\log r|)^\alpha}\int_{|y|=r} f(y) d\sigma_y,
        \]
        applying \cite[Lemma 3.2]{R10}, and since there exists $C>0$ such that for every $r\geq 0$
        \[
        \mathcal{J}(r)\geq \frac{C}{(1+r^2)^{1/4}},
        \]
        we get
        \[
        V(f,f)
        \geq C \left(\int_0^{+\infty} h(r) dr\right)^2
        \geq C \left(\int_0^{+\infty} \frac{1}{(1+r^2)^{1/4}(1+|\log r|)^\alpha}\int_{|y|=r} f(y) d\sigma_y dr\right)^2
        \]
        concluding the proof.
	\end{proof}
\section{The functional setting}\label{Se4}
In this section we study some properties of the space $\mathcal E$ introduced in \eqref{E}. 
Analogously to \cite[Prop. 2.2]{R10}, we have the following property.
\begin{prop}\label{PropA.2}
The space $(\E,\|\cdot\|_{\E})$ is a uniformly convex Banach space. 
	\end{prop}
    \begin{proof}
First we observe that, if $u,v\in\mathcal{E}$,
by H\"older inequality, 
		\begin{equation}\label{uvuv}
			|V(uv,uv)|
            \leq V(|uv|,|uv|)\leq V(u^2,u^2)^\frac{1}{2}V(v^2,v^2)^\frac{1}{2}
            <+\infty.
            \end{equation}
Then also $uv\in\mathcal{E}_1$, in addition to $u^2,v^2\in \mathcal{E}_1$, and so, by Remark~\ref{preHilbert},  we obtain that $\E$ is a vector space on $\R$.

Moreover, if we denote by $\|\cdot\|_{\E_1}$ the norm on $\E_1$ induced by $V$, using again \eqref{uvuv}  we get
\begin{align*}V\big((u+v)^2,(u+v)^2\big)=\|(u+v)^2\|_{\E_1}^2&\leq \big(\|u^2\|_{\E_1}+2\|uv\|_{\E_1}+\|v^2\|_{\E_1})^2\\
	&\leq \big(\|u^2\|_{\E_1}+2\|u^2\|^{1/2}_{\E_1}\|v^2\|^{1/2}_{\E_1}+\|v^2\|_{\E_1}\big)^2\\
	&=\big(\|u^2\|_{\E_1}^{1/2}+\|v^2\|^{1/2}_{\E_1})^4=
	\big(V(u^2,u^2)^{1/4}+V(v^2,v^2)^{1/4}\big)^4.
	\end{align*}
Hence  $u\in \E\mapsto V(u^2,u^2)^{1/4}$ satisfies the triangle inequality and so it is a norm on $\E$.
Then, proceeding along the same lines as \cite[Proposition 2.2]{R10},  we obtain that $(\E,\|\cdot\|_{\E})$ is a uniformly convex Banach space.
\end{proof}
\begin{prop}\label{density}
The space  $C_c^{\infty}(\R^3)$ is dense in $(\E,\|\cdot\|_\E)$.
	\end{prop}
\begin{proof}
Define
    $$\theta_n(x) := \theta(x/n),~~n\ge 1,$$
    where $\theta \in C_c^\infty(\mathbb{R}^3)$ is a smooth cut-off function with
	$\theta(x) = 1$ for $|x| \leq 1$, 
	$\theta(x) = 0$ for $|x| \geq 2$, 
	$0 \leq \theta \leq 1$ everywhere.
    Clearly, $|\nabla \theta_n| \leq C/n$, for some $C>0$. Thus, taking $u \in \mathcal{E}$ and setting $u_n:=\theta_n u$, we 
have that $\{u_n\}\subset H^{1}(\R^3)$. Moreover, let $A_n:=B_{2n}\setminus B_n$.
Since
\[
\|\n u -\n u_n\|_2
   \le C\left(\|u\n \theta_n\|_2 +\|(1-\theta_n)\n u\|_2\right)
   \le \frac{C}{n}\|u\|_{L^6(A_n)} + C\|\n  u\|_{L^2(B_n^c)}=o_n(1),
\]
we deduce that $\nabla u_n\to \nabla u$ in $L^2(\R^3)$. On the other hand,
\begin{align*}
   V((u_n-u)^2,(u_n-u)^2)
   %&=\int_{\R^3}\int_{\R^3}\K(x-y)(u_n-u)^2(x)(u_n-u)^2(y)dxdy\\
   %&=\int_{\R^3}\int_{\R^3}\K(x-y)(\theta_nu-u)^2(x)(\theta_nu-u)^2(y)dxdy\\
     %&=\int_{\R^3}\int_{\R^3}\K(x-y)(\theta_n(x)-1)^2u^2(x)(\theta_n(y)-1)^2u^2(y)dxdy\\
      &=\int_{B_n^c}\int_{B_n^c}\K(x-y)(\theta_n(x)-1)^2u^2(x)(\theta_n(y)-1)^2u^2(y)dxdy\\
       &\le \int_{B_n^c}\int_{B_n^c}\K(x-y)u^2(x)u^2(y)dxdy\to 0.
\end{align*}
This implies that $u_n\to u$ in $\E$.

Therefore, since $ V(u^2,u^2)\leq \|u\|_2^4$ for any $u\in H^{1}(\R^3)$, we have that  $ H^{1}(\R^3)\hookrightarrow \E$ and we conclude by the density  of
		$C_c^{\infty}(\R^3)$ in $H^{1}(\R^3)$. 
	\end{proof}
Moreover we have the following convergence properties.
\begin{lem}\label{lem3.4}
   Let $\{u_n\}$ be a sequence in $\E$. We have that
   \begin{enumerate}[label=(\roman*),ref=\roman*]
       \item \label{weakconv} if $u_n\rightharpoonup u \text{ weakly in } \E$, then
		$\phi_{u_n}\rightharpoonup \phi_{u}$  weakly in $\mathcal{A}$;
        \item \label{stronvconv}  $u_n\to u$ in $\E$ if and only if
    $u_n\to u$ in $D^{1,2}(\RT)$ 
    and $\phi_{u_n}\to \phi_{u}$ in $\A$.
   \end{enumerate}
\end{lem}
 \begin{proof}
		Let us prove (\ref{weakconv}). Note that, if $u_n\rightharpoonup u \text{ weakly in } \E$, then
		$u_n\rightarrow u$ strongly in $L^\tau_{\rm loc}(\R^3)$, for any $\tau\in[1,6)$, Thus, for any $\varphi\in C_c^{\infty}(\R^3)$, it follows by Theorem~\ref{thmE} that
		$$\langle \phi_{u_n},\varphi\rangle_{\mathcal{A}}=4\pi\int_{\R^3}u_n^2\varphi dx \to 4\pi\int_{\R^3}u^2\varphi dx=\langle \phi_{u},\varphi\rangle_{\mathcal{A}},$$
		and, by density, we conclude.

        To prove (\ref{stronvconv}) first observe that, if $u_n\to u$ in $\mathcal{E}$, then, obviously, $u_n\to u$ in $D^{1,2}(\RT)$     and $V(u_n^2,u_n^2)\to V(u^2,u^2)$. Then, since by Theorem \ref{thmE} $4\pi V(u_n^2,u_n^2)= \|\phi_{u_n}\|_{\mathcal{A}}^2$, for any $n\in \N$, using (\ref{weakconv}) we have that, up to a subsequence $\{\phi_{u_n}\}$ converges weakly to  $\phi_u$ in $\mathcal{A}$. Then we get that $\phi_{u_n} \to \phi_u$ in $\mathcal{A}$.
        Vice versa, if
    $u_n\to u$ in $D^{1,2}(\RT)$ 
    and $\phi_{u_n}\to \phi_{u}$ in $\A$ (namely $V(u_n^2,u_n^2)\to V(u^2,u^2)$), then, $\{u_n\}$ is bounded in $\mathcal{E}$ and so, up to a subsequence, it converges weakly in $\mathcal{E}$, and then in $D^{1,2}(\RT)$  to $u$. Then, proceeding as before, we can conclude.
	\end{proof}
Now, using Theorem~\ref{thmE}, we generalise \cite[Proposition 2.2]{SS20}, which is valid for $u\in H^1(\R^3)$, in our setting.
		\begin{prop}\label{pr3}
		For all $u\in\E$, 
\begin{equation}\label{L3}
			\|u\|_3^3\le \frac1\pi\|\phi_u\|_{\mathcal{A}}\|\nabla u\|_2.
		\end{equation}
	\end{prop}
\begin{proof}
Fix $u\in \E$ and
let $\{u_n\}\subset H^{1}(\R^3)$ be an approximating sequence for $u$ in the $\E$-norm defined  as in the proof of Proposition \ref{density}. By \cite[Proposition 2.2]{SS20}, for each $n\geq 1$,   $u_n$ satisfies \eqref{L3}. 
\\
Since $u_n\to u$ a.e. in $\R^3$,  by Fatou's Lemma 
\[	\|u\|^3_3 \leq \liminf_{n\to+\infty}\|u_n\|^3_3.\]
On the other hand, by Theorem~\ref{thmE}, and since $u_n^2\le u^2$, for any $n\ge 1$,   we also have 
\[\|\phi_{u_n}\|^2_{\mathcal{A}}=4\pi\int_{\R^3}\phi_{u_n}u^2_n dx\leq 4\pi \int_{\R^3}\phi_{u}u^2 dx=\|\phi_u\|_{\A}^2.\]
Hence, arguing as in the proof of Proposition \ref{density}, we deduce that $\nabla u_n\to \nabla u$ in $L^2(\R^3)$ and then we get \eqref{L3}. 
\end{proof}

By Proposition \ref{pr3} we deduce easily Theorem \ref{tE36}.

Let's now consider the radial setting.

\begin{proof}[Proof of Theorem \ref{tEr187}]
Let 
	$$ W(x):=\frac{1}{(1+|x|^2)^{1/4}(1+|\log |x||)^\alpha},~\alpha>\frac12.$$
Define the following spaces
	$$L^2_{W}(\R^3):=\Big \{u:\R^3\to\R : \int_{\R^3}W(x)u^2\,dx<+\infty \Big\},$$
    and
	$$H_{W,r}(\R^3):=D_r^{1,2}(\R^3)\cap L^2_{W}(\R^3).$$
	Here, $H_{W,r}(\R^3)$ is a Hilbert space   equipped with the norm
	$$\|u\|_{H_{W,r}}:=\left(\int _{\R^3}(|\nabla u|^2+W(x)u^2)dx\right)^\frac{1}{2}.$$
By Theorem~\ref{Thm1}, we know that $\E_r$ is continuously embedded into $H_{W,r}(\R^3)$.
Now fix $\eps>0$ sufficiently small. By applying \cite[Theorem 1]{SWW07} with $V\equiv W$, $a=-1/2-\eps$, $a_0=\eps$, $Q\equiv 1$, $b= b_0=0$   we infer
\begin{equation*}
		H_{W,r}(\R^3)\hookrightarrow L^\tau(\R^3),\qquad \text{for all } \tau\in 
        \left[\frac{18+4\varepsilon}{7-2\varepsilon},6\right].
	\end{equation*}
By the arbitrariness of $\eps$, we deduce that the optimal range for the continuous embedding is $(18/7, 6]$. 
\\
Arguing as
in the proof of Theorem 1.2 in \cite{R10}, we deduce also the compact embedding.
\end{proof}

In the same spirit of the previous result,   let 
	$$ Z(x):=\frac{1}{(1+|x|)^{\gamma}}, \quad \gamma>\frac12$$
    and define $L^2_Z(\RT)$ and $H_Z(\RT)$ as before, without the radiality assumptions. We have the following embedding.
    \begin{prop}
    The space $\mathcal E$ has continuous embedding into $H_Z(\mathbb R^3)$. In fact, 
the embedding into $L^2_Z(\mathbb R^3)$ is  compact.
    \end{prop}
    \begin{proof}
     If $u\in \mathcal E$, fixed $\varepsilon>0$ there is $R_\varepsilon>0$ such that
     \begin{align*}
        \int_{\mathbb R^3} Z(x) u^2 dx &\le \int_{B_{R_\varepsilon}} u^2 dx + \int_{B_{R_\varepsilon}^c} \frac{(1+|x|^2)^{1/4}(1+|\log |x| |)^\alpha}{(1+|x|)^\gamma} \frac{u^2(x)}{(1+|x|^2)^{1/4} 
(1+|\log |x| |)^\alpha} dx \\
&\le   \int_{B_{R_\varepsilon}} u^2 dx + \varepsilon
\int_{B_{R_\varepsilon}^c}\frac{u^2(x)}{(1+|x|^2)^{1/4} 
(1+|\log |x| |)^\alpha} dx \\
&\le   C \left(\|u\|_6^2 + \varepsilon V(u^2, u^2)^{1/2}\right)\\
&\le  C\left( \|\nabla u\|_{2}^{2} +\varepsilon V(u^2,u^2)^{1/2}\right)
     \end{align*}
     where we used Theorem \ref{thm3.1}.
     Then the continuous embedding holds.
     For the second part, let $u_n\rightharpoonup 0$ in $\mathcal E$, so that in particular $u_n\rightharpoonup 0$
     in $D^{1,2}(\RT)$ and $\{V(u_n^2, u_n^2)\}$ is bounded.
     Given $\varepsilon>0$, there is $R_\varepsilon>0$ such that, arguing as in the previous estimates,
          \[
      \int_{\mathbb R^3} Z(x) u_n^2 dx \le \int_{B_{R_\varepsilon}} u_n^2 dx + \varepsilon C V(u_n^2, u^2_n)^{1/2} =o_n(1) + \varepsilon C,
     \]
and so we can conclude the proof.
    \end{proof}
	\section{The existence result}\label{Se5}
	In this section, we focus on proving Theorem \ref{Thm2}, that is, the existence of the weak solutions of system \eqref{eq1.1} for the case $p\in (3,6)$. 
The  functional associated to problem \eqref{eq1.1} is 
	\begin{equation}\label{I}
		I(u) = \frac{1}{2}\|\nabla u\|_2^2 + \frac{q^2}{4}V(u^2, u^2) - \frac{1}{p}\|u\|^p_p.
	\end{equation}
By Theorem \ref{tE36} the functional $I$ 	is well-defined on $\E$, for each $p\in [3,6]$ and, by Theorem \ref{tEr187}, on $\E_r$, for each $p\in (18/7,6]$. Let us show that $I$ is a $C^2$-functional on both $\E$ and $\E_r$. 
    
First let us give the following preliminary result.

\begin{lem}\label{VFrechet}
		The functional $I_V:= u\in\E\mapsto V(u^2,u^2)\in\R$ is smooth and for each $u,v\in \E$
        \[dI_V(u)[v]=4V(u^2,u v).\]
	\end{lem}
\begin{proof}
We note that 	$I_V$ is the composition of
$S: \mathcal E\to \mathcal E_1$,  $S(u)=u^2$ and $T:\mathcal E_1 \to \R$, $T(v)=V(v, v)$. By Remark~\ref{preHilbert}, $T$ is smooth and so we just need to prove that $S$ is smooth. 
Let $u, v\in \E$. As 
$(u+v)^2- u^2-2uv=v^2$, we get $\|(u+v)^2- u^2-2uv\|_{\E_1}=\|v^2\|_{\E_1}\leq \|v\|^2_{\E}$.
Hence  
$S$ is Fréchet differentiable on $\mathcal E$  with differential at $u\in\mathcal E$ given by $dS(u)[v] = 2uv$
(note that $uv\in \E_1$ by \eqref{uvuv}). 
Now  $dS$ is itself Fréchet differentiable on $\mathcal E$ and, for each $u,v,w\in \E$, $d^2S(u)[v,w]=2vw$ and hence $S$ is smooth.
Finally, by the chain rule, we then get
\[dI_V(u)[v]=dT(S(u))\big[dS(u)[v]\big]=2V(S(u), 2uv)=4V(u^2,uv),\]
completing the proof.
 \end{proof}
We observe now that the first term in the expression of  $I$ in \eqref{I} is clearly smooth in $\E$ and, by Theorem \ref{tE36}, the last one belongs to $C^2(\E)$ if $p\in[3,6]$; analogously, it is $C^2(\E_r)$ if $p\in (18/7,6]$ thanks  to Theorem \ref{tEr187}.
Thus,  from Lemma~\ref{VFrechet} we can state
\begin{prop}\label{IC1}
		For each $p\in[3,6]$, $I\in C^2(\E)$ and   for any $u, v\in \E$
		\begin{equation*}
			I'(u)[v] = \int_{\R^3} \nabla u \cdot \nabla v \, dx + q^2 \int_{\R^3} \phi_u u v \, dx - \int_{\R^3} |u|^{p-2}u v \, dx.
		\end{equation*}
The same result holds on $\mathcal E_r$, for $p\in(18/7,6]$.
	\end{prop}

As an obvious consequence of Proposition~\ref{IC1} we have
	\begin{cor}
	Let $p\in[3,6]$ (resp. $p\in(18/7,6]$) and $u\in \E$ (resp. $u\in \E_r$). Then $(u,\phi_u)$   is a weak solution of the system  \eqref{eq1.1} if and only if	$I'(u)=0$. 	
\end{cor}		
Now we pass to the main  goal of this section. In order to get the existence of a solution, taking into account the results in \cite{CLRT22,DS19}, we use a perturbation approach.
For this purpose, we introduce the following perturbed system by adding a positive small mass term to the system \eqref{eq1.1}. More precisely, for any $\eps>0$ and $3<p<6$, we consider
	\begin{equation}\label{eq3.5}\tag{$\P_\eps$}
\begin{cases}
			-\Delta u+\eps u +q^2\phi u=|u|^{p-2}u\\
			-\Delta \phi+\Delta^2\phi=4\pi u^2
		\end{cases}\text{ in }\RT.   
	\end{equation}
Solutions of \eqref{eq3.5} can be found as critical points of the energy functional $I_\eps:H^1(\R^3)\to \R$, given by
	\begin{equation*}
		I_\eps(u):=\frac12\|\nabla u\|_2^2+\frac{\eps}{2}\|u\|_2^2+\frac{q^2}{4}V(u^2,u^2)%\int_{\R^3}\int_{\R^3}\K(x-y)u^2(x)u^2(y)\,dx\,dy
        -\frac{1}{p}\|u\|^p_p.
	\end{equation*}
By \cite{DS19}, we know that any solution of \eqref{eq3.5} satisfies the Pohozaev type identity
\begin{align*}
     P_\eps(u):=&\frac12\|\nabla u\|_2^2+\frac{3\eps}{2}\|u\|_2^2-\frac{3}{p}\|u\|_p^p\\
		&\quad+\frac{q^2}{4}\int_{\R^3}\int_{\R^3}\frac{5(1-e^{-{|x-y|}})+|x-y|e^{-{|x-y|}}}{|x-y|}u^2(x)u^2(y)\,dx\,dy=0.
\end{align*}
Following \cite{CLRT22}, we define the Nehari-Pohozaev set
	$$\mathcal{M}_\eps:=\Big\{u\in H^1(\R^3)\setminus\{0\}: J_\eps(u)=0 \Big\},$$
	where
	\begin{equation}\label{eq3.20}
\begin{split}
		J_\eps(u):=&2I_\eps'(u)[u]-P_\eps(u)=\frac32\|\nabla u\|_2^2+\frac{\eps}{2}\|u\|_2^2-\frac{2p-3}{p}\|u\|_p^p \\
		&\quad+\frac{3q^2}{4}\int_{\R^3}\int_{\R^3}\frac{1-e^{-{|x-y|}}-\frac{|x-y|}{3}e^{-{|x-y|}}}{|x-y|}u^2(x)u^2(y)\,dx\,dy.
\end{split}
	\end{equation}
	Then, by \cite[Corollary 1.6]{CLRT22} we know that there exists a ground state solution $u_\eps\in H^1(\R^3)$ such that
    \begin{equation}\label{eq3.21}
		0<{m}_\eps:=I_\eps(u_\eps)=\inf_{u\in \mathcal{M}_\eps}I_\eps(u)
        =\inf_{u\in H^1(\R^3)\setminus\{0\} }\max_{t>0}I_\eps\big(t^2u(tx)\big).
	\end{equation}
The following inequalities hold (compare with \cite[Lemma 3.1]{CLRT22}).
	\begin{lem}
		Let $b\ge0$. Then 
		\begin{equation}\label{eq3.24}
			t^3(e^{-\frac{b}{t}}-e^{-b})+\frac{1-t^3}{3}be^{-b}\ge0, ~~\forall t>0,
		\end{equation}
		and
		\begin{equation}\label{eq3.33}
			\frac12(1-e^{-b})-\frac{1}{3}be^{-b}\ge0.
		\end{equation}
	\end{lem}
	In view of the above lemmas, we now intend to show that the energy value $m_\eps$ defined in \eqref{eq3.21} is uniformly bounded by positive constants both from above and from below. This plays a key role when we rule out the vanishing case in the proof of Theorem \ref{Thm2}.
	\begin{lem}\label{lem3.5}
		Let $q>0$ and $p\in (3,6)$. There exists a constant $C>0$ such that for any $\eps\in(0,1)$ it holds
		\begin{equation*}
			C\le m_\eps\le m_1.
		\end{equation*}
		Moreover, the family $\{u_\eps\}_{\eps>0}$ is bounded in $\E$ and bounded away from $0$ in $L^p(\RT)$.
	\end{lem}
	\begin{proof}
		Let $u_1\in \mathcal{M}_1$ be such that $J_1(u_1)=0$ and $I_1(u_1)=\inf_{u\in\mathcal M_1}I_1(u)=:m_1$. Following \cite[Lemma 3.8]{CLRT22}, there exists a unique value $t_\eps>0$ such that $t_\eps^2 u_1(t_\eps \cdot)\in \mathcal{M}_\eps$. Then, we deduce from \eqref{eq3.21} that
		$$m_\eps\le I_\eps\big(t_\eps^2 u_1(t_\eps \cdot)\big).$$
		Thus, since $J_1(u_1)=0$, we obtain that
  \begin{equation*}
			\begin{split}
				m_1-m_\eps
				& \ge I_1(u_1)-I_\eps\big(t_\eps^2 u_1(t_\eps \cdot)\big)
                \\
                &=\frac{1-t_\eps^3}{2}\|\nabla u_1\|_2^2
                +\frac{1-\eps t_\eps}{2}\|u_1\|_2^2
                +\frac{t_\eps^{2p-3}-1}{p}\|u_1\|_p^p 
                \\
				&~~~~+\frac{q^2}{4}\int_{\R^3}\int_{\R^3} \frac{1-e^{-|x-y|} - t_\varepsilon^3(1-e^{-\frac{|x-y|}{t_\eps}})}{|x-y|}u_1^2(x)u_1^2(y)\,dx\,dy
                \\
            &=\left(\frac{1-\eps t_\eps}{2}-\frac{1-t_\eps^3}{6}\right) \|u_1\|_2^2
                +\frac{1}{p}\left[t_\eps^{2p-3}-1+\frac{(1-t_\eps^3)(2p-3)}{3}\right]\|u_1\|_p^pdx
                \\
				&~~~~+\frac{q^2}{4}\int_{\R^3}\int_{\R^3}\frac{t_\eps^3(e^{-\frac{|x-y|}{t_\eps}}-e^{-{|x-y|}})+(1-t_\eps^3)
					\frac{|x-y|}{3}e^{-{|x-y|}}}{|x-y|}u_1^2(x)u_1^2(y)\,dx\,dy.
			\end{split}
		\end{equation*}
		Define $$f_\varepsilon(t):=\frac{1-\eps t}{2}-\frac{1-t^3}{6} \quad\text{and}\quad g(t):=t^{2p-3}-1+\frac{(1-t^3)(2p-3)}{3}\quad\text{for } t>0.$$
        A simple calculation shows that
 $$f_\varepsilon(t_\eps)\ge f_\varepsilon(\eps^{\frac12})=\frac13-\frac13\eps^{\frac32}>0, \text{ for any } \eps\in(0,1), $$
		and
		$$g(t_\eps)\ge g(1)=0, $$
		which, combined with \eqref{eq3.24}, implies that
		\begin{equation}\label{eq3.28}
			m_\eps\le m_1.
		\end{equation}
Let us show that the above inequality implies the boundedness of $\{u_\varepsilon \}_{\varepsilon >0}$ in $\mathcal{E}$.
Since $J_\eps(u_\eps)=0$ and $I_\eps(u_\eps)=m_\eps$, it is easy to derive that
		\begin{equation}\label{eq3.26}
			\begin{aligned}
				m_\eps
                &=I_\eps(u_\eps)-\frac{1}{2p-3}J_\eps(u_\eps)\\
                &=\frac{p-3}{2p-3}\|\nabla u_\eps\|_2^2
				+\frac{(p-2)\eps}{2p-3}\|u_\eps\|_2^2
				%&~~~~+\int_{\R^3}\int_{\R^3}\frac{1-e^{-{|x-y|}}}{|x-y|}u_\eps^2(x)u_\eps^2(y)\,dx\,dy
                +\frac{q^2(p-3)}{2(2p-3)} V(u_\eps^2,u_\eps^2)\\
				&\quad\quad+\frac{q^2}{4(2p-3)}\int_{\R^3}\int_{\R^3}e^{-{|x-y|}}u_\eps^2(x)u_\eps^2(y)\,dx\,dy,
			\end{aligned}
		\end{equation}
and so \eqref{eq3.28}  gives that
the family $\{u_\eps\}_{\eps>0}$ is bounded in $\E$.

Observe, moreover that since $u_\eps\in \mathcal{M}_{\eps}$, by \eqref{eq3.20} and \eqref{eq3.33}, we get that
		\begin{equation}\label{eq3.32}
			\begin{aligned}
				\frac{2p-3}{p}\|u_\eps\|_p^p
				&=\frac32\|\nabla u_\eps\|_2^2+\frac{\eps}{2}\|u_\eps\|_2^2+\frac{3q^2}{8} V(u_\eps^2,u_\eps^2)\\
                &\quad\quad+\frac{3q^2}{4}\int_{\R^3}\int_{\R^3}\frac{\frac12(1-e^{-{|x-y|}})-\frac{|x-y|}{3}e^{-{|x-y|}}}{|x-y|}u_\eps^2(x)u_\eps^2(y)\,dx\,dy\\
				&\ge\frac32\|\nabla u_\eps\|_2^2+\frac{3q^2}{8} V(u_\eps^2,u_\eps^2)
                \\
				&\ge C\left(\|\nabla u_\eps\|_2^2+ V(u_\eps^2,u_\eps^2)\right).
			\end{aligned}
		\end{equation}
 Since $D^{1,2}(\R^3)\hookrightarrow L^6(\RT)$, for  $\lambda = (6-p)/p\in (0,1)$ and by \eqref{L3}, we have
		\begin{equation}\label{eq3.34}
			\begin{aligned}
          \|u_\eps\|_{p}
          &\le\|u_\eps\|_{3}^{\lambda} \|u_\eps\|_{6}^{1-\lambda}\\
          &\le C\left(\|\phi_{u_\eps}\|_{\A}\|\nabla u_\eps\|_2\right)^{\frac{\lambda }{3}}\|\nabla u_\eps\|_2^{1-\lambda } \\
        &\le C\left(\|\phi_{u_\eps}\|_{\A}^2+\|\nabla u_\eps\|_2^2\right)^{\frac{\lambda}{3}} \left(\|\nabla u_\eps\|_2^2+V(u_\eps^2,u_\eps^2)\right)^{\frac{1-\lambda }{2}}\\
   &\leq C \left(\|\nabla u_\eps\|_2^2+V(u_\eps^2,u_\eps^2)\right)^{\frac{3-\lambda }{6}}.
			\end{aligned}
		\end{equation}
Then by \eqref{eq3.32} and \eqref{eq3.34}
        we have 
        \begin{equation*}
        \|\nabla u_\eps\|_2^2+V(u_\eps^2,u_\eps^2)\le C \|u_\eps\|_{p}^p< C\left(\|\nabla u_\eps\|_2^2+V(u_\eps^2,u_\eps^2)\right)^{(2p-3)/3},
        \end{equation*}    that is
        \begin{equation}
            \label{ba0}
            \left(\|\nabla u_\eps\|_2^2+V(u_\eps^2,u_\eps^2)\right)^\frac{2(p-3)}{3} \geq C.
        \end{equation}
   Being $p>3$, the above inequalities imply the boundedness away from $0$ in $L^p(\RT)$ (and also in $\E$) of $\{u_\eps\}_{\eps>0}$.   
        
Finally, by \eqref{eq3.26} and \eqref{ba0} we get $m_\varepsilon\geq C$, concluding the proof. 
\end{proof}
Finally, with the help of the above lemmas, we  can prove Theorem \ref{Thm2}.
\begin{proof}[Proof of Theorem~\ref{Thm2}]
Let $\{\eps_n\}$ be such that $\eps_n \searrow 0$, as $n\to+\infty$.
	Recalling \eqref{eq3.21}, for every $n\in\mathbb{N}$, there exists $u_n:=u_{\eps_n}\in \mathcal{M}_{\eps_n}$, such that $I_{\eps_n}(u_n)=m_{\eps_n}$ and $I^\prime_{\eps_n}(u_n)= 0$. 
   We firstly claim that
\begin{equation*}
\liminf_{n\to+\infty}\sup_{z\in\R^3}\int_{B_1(z)}|u_n|^{3}dx= C>0.
	\end{equation*}
	On the contrary, there exists a subsequence of $\{u_n\}$, still denoted by $\{u_n\}$, such that
	$$\lim_{n\to+\infty}\sup_{z\in\R^3}\int_{B_1(z)}|u_n|^{3}dx=0,$$
	which, together with Lemma \ref{lem3.5},  Theorem \ref{tE36}, and \cite[Lemma I.1]{L84}, indicates that
    \begin{equation*}
		u_n\to 0\quad \text{  in } L^\tau(\R^3), \quad\text{ for all }\tau\in\left(3,6\right).
	\end{equation*}
This gives directly a contradiction with Lemma \ref{lem3.5}. Hence, there exists a sequence $\{z_n\}\subset\R^3$ such that
	\begin{equation*}
\lim_{n\to+\infty}\int_{B_1(z_n)}|u_n|^{3}dx\ge C>0.
	\end{equation*}
Set $v_n:=u_n(\cdot+z_n)$, we obtain that $$\lim_{n\to+\infty}\int_{B_1}|v_n|^{3}dx\ge C>0.$$
Furthermore, noting that $\{v_n\}$ is bounded in $\E$ (since $\{u_n\}$ is bounded in $\E$ by Lemma \ref{lem3.5}
and  $V(u_n^2,u_n^2)=V(v_n^2,v_n^2)$ due to the invariance by translations), up to a subsequence, there holds
 $v_n\rightharpoonup v$  weakly in $\E$,
	and
	$v_n\rightarrow v $      in $L^\tau_{\rm loc}(\R^3)$, for all $\tau\in[1,6).$ So $v\not\equiv0$.

     Since $u_n$ is a critical point of the energy functional $I_{\eps_n}$, for each $n\in \N$, due to the translation invariance, then we conclude that,
for all $\varphi\in C_c^{\infty}(\R^3)$,
	\begin{equation}\label{eq2.4}
		\int_{\R^3}\nabla v_n \cdot\nabla \varphi dx+\eps_n\int_{\R^3}v_n\varphi dx+q^2 \int_{\R^3}\phi_{v_n}v_n\varphi dx=\int_{\R^3}|v_n|^{p-2}v_n\varphi dx.
	\end{equation}
	Applying H\"older inequality and  Theorem \ref{tE36}, we get, as $n\to +\infty$,
	$$\int_{\R^3}\nabla v_n \cdot\nabla \varphi dx\to\int_{\R^3}\nabla v \cdot\nabla \varphi dx ,$$
	$$\left|\eps_n\int_{\R^3}v_n\varphi dx\right|\le \eps_n\left(\int_{\R^3}|v_n|^6 dx\right)^{\frac16}\left(\int_{\R^3}|\varphi|^{\frac65} dx\right)^{\frac56}\to0,$$
	and
	\begin{equation*}
		\begin{aligned}
			&\left|\int_{\R^3}|v_n|^{p-2}v_n\varphi dx-\int_{\R^3}|v|^{p-2}v\varphi dx\right|
			\le \left(\int_{{\rm supp}\varphi}\left||v_n|^{p-2}v_n -|v|^{p-2}v\right|^{\frac{p}{p-1}} dx\right)^{\frac{p-1}{p}}\left(\int_{\R^3}|\varphi|^p dx\right)^{\frac{1}{p}}\to 0.
		\end{aligned}
	\end{equation*}
Furthermore, by Lemma \ref{lem3.4} and Theorem \ref{thmE}, we also have that
	\begin{equation*}
		\begin{aligned}
		&\left|\int_{\R^3}\phi_{v_n}v_n\varphi dx-\int_{\R^3}\phi_{v}v\varphi dx\right|\\
			&\le \int_{\R^3}|\phi_{v_n}-\phi_{v}||v_n||\varphi| dx+\int_{\R^3}|\phi_{v}||v_n-v||\varphi| dx\\
            &\le\left(\int_{{\rm supp}\varphi}|\phi_{v_n}-\phi_{v}|^2dx\right)^{\frac12}\left(\int_{{\rm supp}\varphi}|v_n|^6 dx\right)^{\frac16}\left(\int_{{\rm supp}\varphi}|\varphi|^3 dx\right)^{\frac13}\\
            &~~~~+\left(\int_{{\rm supp}\varphi}|\phi_{v}|^6 dx\right)^{\frac16}\left(\int_{{\rm supp}\varphi}|v_n-v|^2 dx\right)^{\frac12}\left(\int_{{\rm supp}\varphi}|\varphi|^3 dx\right)^{\frac13}\to0, \quad\text{ as } n\to +\infty.
		\end{aligned}
	\end{equation*}
	As a consequence, for any $\varphi\in C_c^{\infty}(\R^3)$, we deduce from \eqref{eq2.4} that, as $n\to +\infty$,
	$$\int_{\R^3}\nabla v \cdot\nabla \varphi dx+q^2 \int_{\R^3}\phi_{v}v\varphi dx=\int_{\R^3}|v|^{p-2}v\varphi dx.$$
By Proposition \ref{density}, the previous identity holds also for every $\varphi\in \E$ and so $v$ is a critical point for $I$, namely there exists a non-trivial weak solution in $\E$ to system \eqref{eq1.1}.
	\end{proof}
 By using a similar reasoning as the above proof of  Theorem~\ref{Thm2} together with Theorem \ref{tEr187}, we can  more directly conclude that system \eqref{eq1.1} has a non-trivial weak solution in $\E_r$ for  $p\in (3,6)$.

\bigskip

{\bf Acknowledgments}
E.C., P.D., A.P., and G.S. are members of INdAM-GNAMPA and are  financed by European Union - Next Generation EU - PRIN 2022 PNRR ``P2022YFAJH Linear and Nonlinear PDE's: New directions and Applications".
E.C., P.D., and A.P. are  partially supported by the Italian Ministry of University and Research under the Program Department of Excellence L. 232/2016 (Grant No. CUP D93C23000100001).
P.D., A.P., and G.S. are partially supported by INdAM-GNAMPA Project 2025 (CUP E5324001950001).
E.C., P.D., and A.P. are partially supported by INdAM-GNAMPA Project 2026 (CUP E53C25002010001).
G.S. is partially supported by Capes, CNPq, FAPDF Edital 04/2021 - Demanda Espont\^anea,
Fapesp grants no. 2022/16407-1 and 2022/16097-2 (Brazil).
L.Y. is supported by the China Scholarship Council (Grant No. 202406660027) and appreciates the warm hospitality from  Politecnico di Bari and Dipartimento di Meccanica, Matematica e Management.

\smallskip

{\bf Data availability} Data sharing not applicable to this article as no datasets were generated
or analysed during the current study.

\smallskip

{\bf Conflict of interest} The authors declare that they have no conflict of interest.

\end{document}